\documentclass[12pt]{amsart}
\usepackage{graphicx} 
\usepackage{amssymb,amsmath}
\usepackage{mathabx}
\usepackage[usenames,dvipsnames]{xcolor}
\usepackage{amsthm}
\usepackage{enumitem, csquotes}
\usepackage{fullpage}
\usepackage{hyperref}
\newcommand{\R}{\mathbb{R}}
\newcommand{\p}{\partial}
\newcommand{\px}{\partial_x}
\newcommand{\pt}{\partial_t}
\newcommand{\K}{\kappa}
\newcommand{\bu}{\bar{u}}

\newcommand{\lp}{\left(}
\newcommand{\rp}{\right)}
\DeclareMathOperator{\I}{Im}
\DeclareMathOperator{\Rea}{Re}

\newtheorem{thm}{Theorem}[section]
\newtheorem{lemma}{Lemma}[section]
\newtheorem{rmk}{Remark}[section]
\newtheorem{prop}{Proposition}[section]

\title{Global Dynamics of small data solutions to the Derivative Nonlinear Schr\"odinger equation}
\author{Allison Byars}
\date{}

\begin{document}
\begin{abstract}
    In this paper, we consider the derivative nonlinear Schr\"odinger
(DNLS) equation. While the existence theory has been intensely studied, properties like dispersive estimates for the solutions have not yet been investigated. Here we address this question for the problem with small and localized data, and show that a  dispersive estimate for the solution holds globally in time. For the proof of our result we use vector field methods combined with the \emph{testing by wave packets method}, whose implementation in this problem is novel.
\end{abstract}

\maketitle

\section{Introduction}

We consider the Cauchy problem for the derivative nonlinear Schr\"odinger
(DNLS) equation
\begin{equation}
\begin{cases}
iu_t+u_{xx}=-i\partial_x(|u|^2u)\\
u(0,x)=u_0(x),
\end{cases}
\tag{DNLS}
\label{dnls}
\end{equation}
 where the unknown is a complex valued function, $u:\mathbb{R}\times \mathbb{R}\rightarrow \mathbb{C}.$ \eqref{dnls} models the propagation of large-wavelength Alfv\'en waves in plasma. There are also multiple other physical phenomena that are modeled by this equation, which makes this problem even more interesting. The equation was first derived in 1974 by Mj{\o}elhus \cite{mjoelhus1974application} using the reductive perturbation method. Though the derivation is not the focus of this paper, it is helpful to see it in order to understand more about the model and the stability of the Alfv\'en waves.  We refer the interested reader to the following references as well as the references within \cite{cite:champeaux, jenkins2019derivative,cite:physical2,cite:physical3, cite:physical4}. 

The \eqref{dnls} equation is a dispersive equation, which has the dispersion relation $\omega (\xi)=-\xi^2,$ where the group velocity of waves is given by $\omega'(\xi)=-2\xi$. This depends on the frequency $\xi$, which yields the dispersive character.
  
\eqref{dnls} also admits the following scaling law: If $u$ is a solution,  then for $\lambda>0,$  so is
\begin{equation*}
    u_\lambda (t,x):= \sqrt{\lambda}u(\lambda^2t, \lambda x).
\end{equation*}
This scaling gives the critical Sobolev space as $L^2(\R)$, which plays an important role as a local well-posedness threshold.

Another interesting property of \eqref{dnls} is that it is completely integrable, and has an associated Lax Pair 
 \begin{align*}
    L(\kappa;u)=\begin{bmatrix}
1 & 0 \\
0 & -1 
\end{bmatrix}\begin{bmatrix}
\K-\px & \sqrt{\K}u \\
i\sqrt{\K}\bar{u} & \K+\px 
\end{bmatrix}=\begin{bmatrix}
\K-\px & \sqrt{\K}u \\
-i\sqrt{\K}\bar{u} & -(\K+\px) 
\end{bmatrix}
\end{align*}
and 

\begin{align*}
   M(\K)= \begin{bmatrix}
        2i\K^2-\K|u|^2 & 2i\K^{\frac{3}{2}}u-\K^{\frac{1}{2}}|u|^2u+i\K^{\frac{1}{2}}u_x\\
        2\K^{\frac{3}{2}}\bu +i\K^{\frac{1}{2}}|u|^2\bu -\K^{\frac{1}{2}}\bu_x & -2i\K^2+\K|u|^2
    \end{bmatrix}
\end{align*}
where 
\begin{align*}
    \frac{d}{dt}L\Phi=[M,L]\Phi.
\end{align*}
This Lax Pair was introduced by Kaup and Newell \cite{cite:KaupNew} and further explored in  \cite{cite:WPvisanH16}. We will not be using this here, but it is very helpful in many other situations. \eqref{dnls} also has an inverse scattering transform \cite{jenkins2018global} and an infinite number of conservation laws, three of which are the mass,
\begin{equation*}
    M(u):= \int |u(x)|^2\, dx,
\end{equation*}
the momentum,
\begin{equation*}
    P(u):=\int \I(\bar{u}\px u)-\frac{1}{2}|u|^4 \, dx,
\end{equation*}
and the energy, 
\begin{equation*}
     E(u):=\int |u_x|^2-\frac{3}{2}|u|^2 \I(\bar{u}u_x)+\frac{1}{2}|u|^6 \, dx.
\end{equation*}

It also has a Hamiltonian structure given by the operator $\mathcal{J}=\px$ with $P$ as the Hamiltonian. This means we can generate the \eqref{dnls} equation using the functional derivative of the momentum: 
\begin{align*}
    u_t=\mathcal{J}\left(\frac{\delta P(u)}{\delta \bu}\right).
\end{align*}
While interesting, the complete integrability does not play any role in the 
present work.

\medskip

When given a PDE, the first natural question to address is the well-posedness of the system in suitably  chosen Sobolev spaces. 
As stated above, the critical Sobolev space for \eqref{dnls} is $H^0(\R)=L^2(\R)$.  This means that we expect to get well-posedness above this threshold, so in $H^s$ for $s\geq 0$ and ill-posedness for $s<0.$ However, this does not always work out exactly as we would expect.  By looking at the equation, we can see \eqref{dnls} is technically a semilinear equation. However, for $H^s$ with $s<\frac{1}{2}$, we do not have Lipschitz dependence on the initial data \cite{cite:Bialess12}, so below that threshold, the problem behaves more like a quasilinear problem.  This means that for $s>\frac{1}{2}$, we can most likely prove well-posedness using a fixed point argument. In fact, well-posedness for $s>\frac{3}{2}$ has been well studied \cite{cite:WP>3/2,cite:WP>3/2.2}, and the case $s\geq\frac{1}{2}$ was proved in '99 \cite{cite:Taks12} using $X^{s,b}$ spaces. 
The case $s=\frac{1}{2}$ was studied further in \cite{bahouri2022global} where they prove global well-posedness in $H^{\frac{1}{2}}$.
The ill-posedness for $s<0$ can be seen from the self similar solutions constructed in \cite{cite:selfsim1,cite:selfsim2}.
Another related classical result is the work of Hayashi and Ozawa in \cite{hayashi1994modified}.  In this paper, the authors proved the existence of a modified wave operator.  In Corollary 1.1, we see the decay rate of the solution they got matches the one we obtain in Theorem \ref{t:main}.

 The remaining gap $s \in [0,1/2)$ was only recently settled by  Harrop-Griffiths et al. \cite{cite:1} who proved that \eqref{dnls} is globally well-posed in $L^2(\R)$, and more generally in $H^s$ for $0\leq s\leq \frac{1}{2}.$ They used the second generation method of commuting flows, where they take advantage of local smoothing and tightness estimates. This result closes the gap and finishes the study of well-posedness for \eqref{dnls}. 

Once we have the well-posedness theory, the next natural questions to ask are: What does the solution look like on long time scales? Does it decay like the solution to the linear equation? Are there solitons?

 Work on the dispersive decay of \ref{dnls} has been done in \cite{hayashi1997asymptotic,hayashi1998asymptotic}. However, in these papers, they use gauge transformations and pseduo conformal transforms to simplify \ref{dnls} and transform it into a semilinear problem  

In this paper, we prove, under some mild conditions on the initial data, that the solution to \eqref{dnls} decays at the same rate as the solution to the linear equation, globally in time.  We emphasize that in contrast with \cite{hayashi1997asymptotic,hayashi1998asymptotic}, we treat the nonlinearity as quasilinear, without transforming it to a semilinear system.  This allows us to improve on the regularity assumptions, only taking the initial data to be small in $L^2$ and localized in $H^1,$ whereas they take the initial data in $H^{\frac{3}{2}+\delta}$ and localized in $H^1.$ 

This result is a contribution to the study of the \textit{Soliton Resolution Conjecture}, which roughly states that under some ideal conditions on the initial data, every solution to a dispersive partial differential equation can be written as the sum of solitons and a dispersive part. In this result, for the class of initial data we consider, there are no solitons, and hence it is reasonable to expect that we must have a global in time dispersive estimate for the solution.

\subsection{The main result} In order to state the main result, we first need to introduce the vector field $L$, which is the  pushforward of $x$ along the linear Schr\"odinger flow, and thus it commutes with the linear flow. This is defined as
\begin{equation}
\label{def:L}
    L:=x+2it\px .
\end{equation}
The vector field $L$ is used to measure the initial data localization as well as its effect on the solution later in time.
Our main theorem is as follows:
\begin{thm}
\label{t:main}
    Let $u$ be a solution to \eqref{dnls} with small and localized data, i.e. for $\epsilon \ll 1,$
    \begin{equation}
        \|xu_0\|_{H^1}+\|u_0\|_{ L^2}\leq \epsilon
        \label{initials}
    \end{equation}
Then we have the following bound on $Lu$,
\begin{equation}
    \|Lu\|_{H^1_x}\lesssim \epsilon\, \langle t \rangle^{C\epsilon^2},
    \label{Lubound1}
\end{equation}
and the dispersive bounds
\begin{equation*}
    \|u(t,x)\|_{L^\infty_x} +  \|u_x(t,x)\|_{L^\infty_x}\lesssim \epsilon \, \langle t\rangle^{-\frac{1}{2}} \text{ for all } t\in \R.
    \label{dispersiveestimate1}
\end{equation*}
\end{thm}
Here, $\langle \cdot \rangle$ are the usual Japanese brackets
\begin{align*}
    \langle t\rangle:=\left(1+t^2\right)^{\frac{1}{2}}.
\end{align*}

\begin{rmk}
    The constant $C$ in \eqref{Lubound1}  is  a large universal constant, which in particular does not depend on $\epsilon$.
\end{rmk}

The proof is completed in several steps. We first prove energy estimates for $u$ and also for $Lu$. If these estimates were 
uniform in time, then dispersive decay for the solutions would follow from  Klainerman-Sobolev type inequalities, as seen below for the linear case. However, the estimates we obtain for  $Lu$ do exhibit a slow growth in time, which 
 prevents such a direct argument.  We address this difficulty by making use of the testing by wave packets method of Ifrim and Tataru \cite{cite:WavePax}. This method  allows us to construct an \emph{asymptotic profile} for the solution, which in turn is an approximate solution for an \emph{asymptotic equation}. The global dispersive bounds are propagated in time exactly using this asymptotic equation.

 \subsection{Solitons}

The solitons provide the simplest obstruction to a global dispersive estimate for the nonlinear equation so it is natural to consider the following questions:

\begin{itemize}
   \item Do solitons exist?
     \item Do solitons exists within our class of initial data?
\end{itemize}

Broadly speaking, there are two cases to consider when proving dispersive decay bounds for a nonlinear equation. These two cases will depend on the existence of solitons. Solitons are stationary waves which do not change shape over time. 
If the equation admits solitons, then we have no hope of having a dispersive bound globally in time.  The existence of solitons depends completely on the assumptions we make on the initial data. 

1. If solitons exist, we do not expect the nonlinear solution to decay like the linear solution globally in time.  In this case, the best we can hope for is that the solution to the nonlinear equation decays like the solution to the linear equation up to some time $T$, depending on the initial data.  Then we would want to prove that the solitons emerge at that timescale, $T,$ and therefore the bound is optimal.  Some examples of this can be seen in  \cite{cite:kdv, cite:ILW, cite:BO} for KdV, Intermediate Long Wave, and Benjamin-Ono, respectively.  

2. If solitons do not exist, then we expect the nonlinear solution to decay like the linear solution globally in time.

We know that the \eqref{dnls} equation does have solitons.  From \cite{cite:1}, we know that the family of solitons for \eqref{dnls} takes the following form.  For $\theta\in (0,\frac{\pi}{2})$, we have the initial data
\begin{align*}
    q_0(x,\theta):=&\sqrt{2\sin (2\theta)}\frac{[\cos(\theta)\cosh(x)-i\sin(\theta)\sinh(x)]^3}{[\cos^2(\theta)\cosh^2(x)+\sin^2(\theta)\sinh^2(x)]^2}e^{-ix\cot(2\theta)}\\
    =&\sqrt{2\sin (2\theta)}\frac{\cosh^3(x-i\theta)}{|\cosh(x-i\theta)|^4}e^{-ix\cot(2\theta)},
\end{align*}
with  the soliton solution given by 
\begin{align*}
    q(t,x,\theta)=q_0(x+2\cot(2\theta)t,\theta)e^{it\csc^2(2\theta)}.
\end{align*}
The solitons can also be rescaled and translated. 

Here we choose the conditions on the initial data so that we are in the second case. It turns out there are no solitons which are small in $L^2$ and simultaneously localized in $H^1$.  
In \cite{cite:1}, they show that the $L^2$ norm of this soliton is 
\begin{align*}
    \|q_0\|_{L^2_x}^2=\|q\|_{L^2_x}^2=8\theta,
\end{align*}
which can be small in the critical norm $L^2$. The $L^2$ norm of $xq_0$ is also comparable to $\theta$. On the other hand the $L^2$ norm of the derivative of $q_0$ has size $1/\theta$, so the $L^2$ norm of $(xq_0)'$ has size $\sim\theta+ \frac{1}{\theta}$. In particular we have the scale invariant bound
\[
 \Vert xq_0\Vert_{H^1}\Vert q_0\Vert_{L^2}\geq 1,
\]
which is true not only for the above $q_0$ but also for all of its rescaled versions.  This shows that indeed, it is not possible for the solitons to be localized in $H^1$ and small in $L^2$.

\subsection{Acknowledgements}
\label{ss:history}
The author would like to thank Mihaela Ifrim (University of Wisconsin-Madison) for proposing this problem and for many illuminating discussions and helpful insights. The author was supported partially by NSF grant DMS-1928930 as well as NSF DMS-2037851.

\section{The Linear Case}
We will first recall the analysis of the dispersive decay properties for the linear equation, which will provide some insight into what we want to look for in the nonlinear case. 
Here we work with the linear Schr\"odinger equation, which is given by 
\begin{equation}\begin{cases}
    iu_t+u_{xx}=0\\
   u(0,x)=u_0(x).
\end{cases}
\label{Lineareq}
\end{equation}

One way to phrase the dispersive decay is in the form of an $L^1$ to $L^\infty$ bound  as follows: 
\begin{prop} For solutions $u$ to \eqref{Lineareq}, we have the following dispersive bound.
    \begin{equation*}
\|u(t,x)\|_{L^\infty_x}\lesssim t^{-\frac{1}{2}}\|u_0\|_{L^1}
\end{equation*}
for all $t\in \R$.
\end{prop} 
\begin{proof}
To prove the proposition we can explicitly find the fundamental solution and use that to find the decay estimate. Here we are using the following definition of the Fourier transform:
\begin{align*}
    \hat{f}(\xi):=\frac{1}{\sqrt{2\pi}}\int_{\R}e^{-ix\xi}f(x)\, dx.
\end{align*}
To solve the linear equation, we first take the Fourier transform to turn the PDE into a linear ODE:
\begin{align*}
    \hat{u_t}(\xi)+i\xi^2\hat{u}(\xi)=0,
\end{align*}
which has the solution 
\begin{align*}
    \hat{u}(t,\xi)=\hat{u_0}(\xi)e^{-i\xi^2t}.
\end{align*}
Therefore, we get the fundamental solution
\begin{align*}
    u(t,x) \approx&u_0*t^{-\frac{1}{2}} e^{-i\frac{x^2}{4t}}.
\end{align*} 
  Then, we can apply Young's inequality to obtain 
\begin{align*}
 \Vert u\Vert_{L^\infty}=  t^{-\frac{1}{2}} \| e^{-i\frac{x^2}{4t}}*u_0(x)\|_{L^\infty}\lesssim t^{-\frac{1}{2}}\| e^{-i\frac{x^2}{4t}}\|_{L^\infty}\|u_0\|_{L^1}\lesssim t^{-\frac{1}{2}}\|u_0\|_{L^1},
\end{align*}
which gives the needed bound.
\end{proof}

While very simple, the above argument is not very helpful when considering the nonlinear problem.  We now provide a second approach based on energy estimates which will serve as a guide for the proof of the nonlinear case. Here we will assume that the initial data is small and localized in $L^2$.  We  formulate the decay result as follows:

\begin{prop} For solutions $u$ to \eqref{Lineareq} with $\|u_0\|_{L^2}+\|xu_0\|_{L^2}\leq \epsilon$, we have the following dispersive bound:
    \begin{equation}
    \|u(t,x)\|_{L^\infty_x}\lesssim \epsilon t^{-\frac{1}{2}}.
\end{equation}
\end{prop}
\begin{proof}
We use the operator $L$, defined in \eqref{def:L}. This operator has the following properties:
\begin{align*}
    \left[ \partial_x , L\right]=1,\quad \left[ i\partial_t+\partial_x^2, L\right] =0.
\end{align*}
If $u$ solves the linear equation \eqref{Lineareq}, then $Lu$ solves it too. This can also be seen by using the scaling derivative as follows  
\begin{align*}
     \frac{d}{d\lambda}u_\lambda(t,x)|_{\lambda=1}
     =&\frac{1}{2}u+xu_x+2tiu_{xx}=\partial_x ( x u + 2 ti u_{x} ) -\frac{1}{2}u.
\end{align*}
This is a solution of \eqref{Lineareq}, and therefore $Lu=xu+2tiu_x$ will also be a solution.  

Using conservation of mass, we see that $\|u\|_{L^2}\leq \epsilon$, and since $Lu$ is also a solution to the linear equation, we also have $\|Lu\|_{L^2}=\|Lu(0)\|_{L^2}=\|xu_0\|_{L^2}\leq \epsilon$. So we have
\begin{equation}
\label{L-estimate}
\Vert  u(t)\Vert_{L^2} + \Vert L u(t)\Vert_{L^2} \leq \epsilon.
\end{equation}
    
The main idea now is to write $\frac{d}{dx}|u|^2$ in terms of the $L^2_x$ norms of $u$ and $Lu$, which we can control.  By applying $L$ to $u$ and solving for $u_x,$ we get
 \begin{align*}
u_x= \frac{1}{2t i}[Lu-xu].
\end{align*}
Then, we obtain 
\begin{align*}
    \frac{d}{dx}|u|^2
    =u\bu_x+\bu u_x=\frac{1}{2t i}[\bu Lu -u\widebar{Lu}].
\end{align*}
Integrating this gives
\begin{align*}
 \|u\|_{L^\infty}^2  &\lesssim\frac{1}{t}\int_{\R} |u||Lu| \, dx \lesssim \frac{1}{t}\|Lu\|_{L^2}\|u\|_{L^2}
\end{align*}
Combining this with the energy bounds \eqref{L-estimate} we obtain the pointwise bound 
\begin{equation}
\|u\|_{L^\infty}\lesssim \epsilon t^{-\frac{1}{2}}.
\end{equation}
\end{proof}
We remark on an intermediate step in the above proof, which is to establish the estimate
\begin{equation}
\|u(t)\|_{L^\infty}^2 \lesssim
\frac{1}{t} \|u(t)\|_{L^2} \|Lu(t)\|_{L^2}.
\end{equation}
This is what we call a \emph{vector field bound}, or a Klainerman-Sobolev type inequality. It no longer depends on the fact that $u$
solves the linear Schr\"odinger equation, which is why in some cases, it also can be used for the nonlinear problem.

\section{Nonlinear Analysis: an outline of the proof\label{s:proofoutline}}
The vector field $L$ helped us prove the dispersive decay bound for the linear equation, so we might expect either $L$ or a nonlinear counterpart to $L$, called $L^{NL}$, will help us with the nonlinear equation.  In the analysis of the cubic NLS flow in \cite{cite:NLS} the same operator $L$ was used. On the other hand, in \cite{cite:kdv, cite:BO} the nonlinear counterpart $L^{NL}$, defined as the scaling derivative of a solution $u$ was the key to proving the dispersive estimate. Finally, in \cite{cite:WavePax} a more systematic approach was introduced in order to construct $L^{NL}$ where there is no scaling symmetry.  However, in our case, because of the nature of the equation, we revert to using the linear $L$ which will turn out to suffice.  We remark that in our setting one can also define  a nonlinear operator $L^{NL}$ as the scaling derivative of the solution but using it does not give significant progress towards the proof.

\medskip

The main idea of the proof is to use the testing by wave packets method, which was introduced by Ifrim and Tataru in \cite{cite:NLS}, and is explained in general in \cite{cite:WavePax}. In both of these papers, as well as in several others, this method  has been used to prove in tandem both global well-posedness and dispersive decay estimates.  The idea here is to construct  well-prepared  approximate solutions, called wave packets, for the linear equation, which are localized in both space and frequency, and travel along rays. These are then used in order to construct a good notion of an asymptotic profile as the  $L^2$ inner product of the wave packet with our solution.  We will prove bounds for the asymptotic profile, via an appropriate asymptotic equation, and then these bounds will help us get bounds for our solution.

Our proof of Theorem \ref{t:main} consists of the following steps:

\medskip
\noindent { \bf I. Find the wave packets.} In this step we will introduce the wave packets ${\Phi}_v$ associated with each ray $x=vt$, which can be defined in a canonical way.  These wave packets are nice approximate solutions to the linear equation, and will be very useful to help prove bounds for the nonlinear solution. 
We will then define the asymptotic profile 
\[
\gamma(t,v):=\int u(t,x) \widebar{\Phi}_v(x)\, dx
\]
as the $L^2_x$ inner product with our solution. This will be done in Section \ref{wavepax}.

\medskip
{\bf II. Difference bounds.}
The next step is to show that the
asymptotic profiles provide a good description of our nonlinear solution $u$ as $t \to \infty$. Precisely, we will prove the following Lemma:
\begin{lemma}
\label{gammalemma}
For $u$ a solution to \eqref{dnls}, and $\gamma$ as defined above, we have
    \begin{equation}
     \|\gamma\|_{L^\infty}\lesssim t^{\frac{1}{2}}\|u\|_{L^\infty}, \,\,\,\,\|\gamma\|_{L^2_v}\lesssim \|u\|_{L^2_x}, \,\,\,\, \|\partial_v \gamma \|_{L^2_v}\lesssim \|Lu\|_{L^2_x}.
        \label{gammabnds}
    \end{equation}
Also,
\begin{align}\label{v gamma L^infty}
  \|v\gamma\|_{L^\infty}\lesssim t^{-\frac{3}{4}}\|Lu\|_{L^2_x}+ t^{\frac{1}{2}}\|u_x\|_{L^\infty}  ,  \,\,\,\,\ \|\langle v \rangle^{\frac{k}{2}}\gamma\|_{L^\infty}\lesssim (\|Lu\|_{L^2_x} + \|u\|_{H_x^{k}}).
\end{align}
We have the following spatial difference bounds,

\begin{align}
 \|u(t,vt)-t^{-\frac{1}{2}}e^{i\phi(t,vt)}\gamma(t,v)\|_{L^\infty}\lesssim& t^{-\frac{3}{4}}\|Lu\|_{L^2_x}
      \label{spatialbnds}\\
    \|u(t,vt)-t^{-\frac{1}{2}}e^{i\phi(t,vt)}\gamma(t,v)\|_{L^2_v}\lesssim& t^{-1}\|Lu\|_{L^2_x},
    \label{spatialL^2}
\end{align}
spatial bounds for the derivative of our solution,
\begin{align}
 \left\|u_x(t,vt)-\frac{i}{2}t^{-\frac{1}{2}}e^{i\phi(t,vt)}v\gamma(t,v)\right\|_{L^\infty}\lesssim& t^{-\frac{3}{4}}\left[\|Lu\|_{L^2_x}+\|L(u_x)\|_{L^2_x}\right]
   \label{spatialu_xbnds}
\end{align}
and Fourier bounds,
\begin{align}
  \|\hat{u}(t,\xi)-e^{-it\xi^2}\gamma(t,2\xi)\|_{L^\infty}  \lesssim
  & t^{-\frac{1}{4}}\|Lu\|_{L^2_x}
  \label{fourierbnds1}\\
    \|\hat{u}(t,\xi)-e^{-it\xi^2}\gamma(t,2\xi)\|_{L^2_\xi}\lesssim
    & t^{-\frac{1}{2}}\|Lu\|_{L^2_x}.
    \label{fourierbnds}
\end{align}
    \label{diffLemma}
\end{lemma}
\begin{rmk}
    Note that the Fourier difference bounds and the spatial $L^2$ bounds are not essential to the proof of the main theorem. However, they help to show that the solution is approximated well by the asymptotic profile $\gamma$, both on the physical side and the Fourier side.
\end{rmk} 

This lemma will allow us to transfer uniform bounds between the solution $u$ and its asymptotic profile $\gamma$.
The proof is given in Section \ref{difference}.

\medskip

{\bf III. The asymptotic equation.}
The key to proving uniform bounds for $\gamma$ as the time goes to infinity is to establish  approximate ODE dynamics for $\gamma$, which we call the asymptotic equation. The precise result is given by the following lemma:
\begin{lemma}
\label{l:approx}
Let $u$ be a solution to \eqref{dnls}. Then $\gamma$ satisfies
\begin{align}
\label{gamma asym}
    i\gamma_t(t)=\frac{v}{2}t^{-1}(|\gamma|^2\gamma)-R(t,v),
\end{align}
where the remainder $R$ satisfies the bounds
\begin{equation}
\begin{aligned}
     \|R\|_{L^\infty}\lesssim & \ t^{-\frac{5}{4}}\|Lu\|_{L^2_x}+\|u\|_{L^\infty}^3 +t^{-\frac{3}{4}}\|u\|_{L^\infty}^2 \|Lu\|_{L^2_x}+  \|u\|_{L^\infty} \|u_x\|_{L^\infty}t^{-\frac{1}{4}}\|Lu\|_{L^2_x}\\
    &+ t^{-\frac{3}{4}}\|u\|_{L^\infty}\|Lu\|_{L^2_x}\lp t^{-\frac{3}{4}}\|Lu\|_{L^2_x}+ t^{\frac{1}{2}}\|u_x\|_{L^\infty}\rp.\label{Rbound0}\\
    \end{aligned}
    \end{equation}
\begin{equation}
  \begin{aligned}  
 \|vR\|_{L^\infty}\lesssim &    t^{-\frac{5}{4}}\lp\|Lu\|_{L^2_x} +\|Lu_x\|_{L^2_x}  \rp +t^{-\frac{5}{4}}\|u\|_{L^\infty}^2 \|Lu\|_{L^2_x}+\|u\|_{L^\infty}^2 \|u_x\|_{L^\infty}\\\
    & + \|u\|_{L^\infty}\left[ t^{-\frac{3}{2}}\|Lu\|_{L^2}^2+t^{-\frac{3}{4}}\|u_x\|_{L^\infty}\|Lu\|_{L^2_x}+ t^{-\frac{1}{4}}\|u_x\|_{L^\infty}\|Lu_x\|_{L^2_x}\right] \\
    & +(t^{-\frac{3}{4}}\|Lu\|_{L^2_x}+ t^{\frac{1}{2}}\|u_x\|_{L^\infty})\|u\|_{L^\infty}\lp t^{-\frac{5}{4}}\|(Lu)_x\|_{L^2_x}+t^{-\frac{3}{4}}\|Lu_x\|_{L^2_x} \rp 
    \label{Rbound1}. 
    \end{aligned}
   \end{equation}
\label{ODELemma}
\end{lemma}
The proof of this Lemma is given in Section \ref{ODE}.

\medskip

{ \bf IV. The bootstrap setup. }
In order to complete the proof of our main result, it will be necessary to use a bootstrap argument.   We will make the bootstrap assumptions
\begin{align}
\label{bootstrap}
    \|u\|_{L^\infty}\leq D\epsilon  \langle t\rangle^{-\frac{1}{2}}, \qquad  \|u_x\|_{L^\infty}\leq D\epsilon  \langle t\rangle^{-\frac{1}{2}},
\end{align}
where $D$ is a large universal constant  to be be chosen later.
Then our task will be to improve 
this constant, under the assumption that $\epsilon$ is small enough. This can be done on an arbitrarily large time interval $[0,T]$ where the solution $u$ exists, where at the end we can let $T \to \infty$ via a continuity argument.

\medskip

{ \bf V. Energy bounds for $u$ and  $Lu$.} These energy bounds are 
critical in order to both estimate the approximation errors in Lemma~\ref{l:approx} and the asymptotic equation errors in Lemma~\ref{gammalemma}. 
They are as follows:
\begin{lemma}
 \label{l:energy est u}
    Let $u$ be  a solution to the DNLS equation with small initial data, $\Vert u_0\Vert_{H^k}\leq \epsilon \ll 1$, for some $k\geq 0.$ Then the solution satisfies the global energy bounds
\begin{equation}
\Vert u\Vert_{H^k}\lesssim \epsilon.
\end{equation}
\end{lemma}

\begin{lemma}
    \label{LuBound}
    Let $L$ be the operator defined in \eqref{def:L}. Then, for any solution $u$ to \eqref{dnls} satisfying \eqref{initials} and \eqref{bootstrap}, we have
    \begin{align*}
    \|Lu\|_{L^2_x}\lesssim & \ \epsilon\langle t\, \rangle^{\frac{D^2\epsilon^2}{2}},\\
        \|L(u_x)\|_{L^2_x}\lesssim&  \ \epsilon\, \langle t\rangle^{\frac{D^2\epsilon^2}{2}}.
\end{align*}
\end{lemma}
We remark here that while the energy bounds for $u$ are uniform in time, a slight loss in the energy bounds for $Lu$ is necessary. The proof of this lemma is given in Section \ref{boundsLU}. 

\medskip

{ \bf VI. Closing the bootstrap argument and the conclusion of the proof. } The final step in our analysis is to improve the bootstrap bound. We state this in the following lemma, whose proof uses
all of the previous lemmas.

\begin{lemma}
\label{finishingprooflemma}
    Let $u$ be a solution to \eqref{dnls} satisfying \eqref{initials} and \eqref{bootstrap}. Then we have 
\begin{align*}
     \|u\|_{L^\infty}\leq E\epsilon  \langle t\rangle^{-\frac{1}{2}}, \qquad 
          \|u_x\|_{L^\infty}\leq F\epsilon  \langle t\rangle^{-\frac{1}{2}},
\end{align*}
    where $E,F<D$.
\end{lemma}
This closes the bootstrap argument, completing the proof of Theorem \ref{t:main}.

\section{The asymptotic profile and the asymptotic equation}
\label{s:asym profile}
This section contains  our implementation of the wave packet testing method.  In particular, we construct the wave packets, define the asymptotic profile $\gamma$, prove the approximation  Lemma~\ref{gammalemma} and establish the asymptotic equation  in Lemma~\ref{l:approx}. 

\subsection{Finding the wave packets\label{wavepax}}

Recall that the dispersion relation for \eqref{dnls} is $a(\xi)=-\xi^2$, with group velocity $v=-2\xi.$  We will consider, for each $v$, wave packets traveling along the ray $x=vt$, which corresponds to the frequency $\xi_v = - v/2$.   
We consider wave packets of the form
\begin{align}
\Phi_v(t,x)= e^{i\phi (t,x)}\chi \left(\frac{x-vt}{\sqrt{t}}\right),
   \label{O(1/t)eqn}
\end{align}
where $\chi$ is a Schwartz function with 
\[
\int \chi(y)\, dy=1,
\]
and the phase $\phi(t,x)=\frac{x^2}{4t}$ is 
the same as the phase of the fundamental solution for the linear Schr\"odinger flow. 
In particular this guarantees that 
\begin{align*}
   ( i\pt+\px^2)\Phi_v=O\left(\frac{1}{t}\right)
\end{align*}
for $t>0$ sufficiently large.
We can see by a direct computation that  we get 
\begin{align}
 ( i\pt +  \px^2)\Phi_v
 =&\frac{e^{i\phi}}{2t}\px \left[i(x-vt)\chi\left(\frac{x-vt}{\sqrt{t}}\right)+2t^{\frac{1}{2}}\chi'\left(\frac{x-vt}{\sqrt{t}}\right)\right].
 \label{LinPhi}
\end{align}
We note that the $\sqrt{t}$ spatial scale for the wave packet is chosen because it is exactly the localization needed in order for the wave packets to stay coherent on dyadic time scales. For more information on how to find this coherence time, see \cite{cite:WavePax}.

We will think of $\Phi_v$ as a good approximation of the solution to the linear Schr\"odinger equation.  We then define the asymptotic profile as follows:
\begin{align*}
    \gamma(t,v) :=\int u(t,x) \bar{\Phi}_v(t,x) \, dx.
\end{align*}

\subsection{Preliminary facts about $\gamma$}
Before we begin the proof of  Lemma ~\ref{diffLemma}, let us make some useful observations  about $\gamma$. First,
we can write $\gamma$ as a convolution.
Define $w:=e^{-i\phi}u(t,x)$. Then we have
\begin{align*}
    \gamma =&\int u \widebar{\Phi_v}\, dx=\int u e^{-i\phi (t)}\overline{\chi \left(\frac{x-vt}{\sqrt{t}}\right)} \, dx,
\end{align*}
which leads to
\begin{align}
   t^{-\frac{1}{2}}\gamma=w(t,vt)*t^{\frac{1}{2}}\chi(vt^{\frac{1}{2}}). 
   \label{gammaconvolution}
\end{align}

This will be helpful for the spatial difference bounds.  For the Fourier bounds, we will need the following, using Plancherel's identity:
\begin{align*}
    \gamma(t,v)=\int \hat{u}(t,\xi) \bar{\hat{\Phi}}_v(t,\xi)\, d\xi.
\end{align*}
In order to write the right hand side as a convolution, we will need to consider $\hat{\Phi}_v$ evaluated at $\frac{\xi}{2}$.
\begin{align*}
\hat{\Phi}_v\left(t,\frac{\xi}{2}\right)=&\frac{1}{\sqrt{2\pi}}\int e^{\frac{-ix\xi}{2}}e^{\frac{ix^2}{4t}}\chi\left(\frac{x-vt}{\sqrt{t}}\right)\, dx\\
     =&t^{\frac{1}{2}}\frac{1}{\sqrt{2\pi}}e^{\frac{-it\xi^2}{4}}e^{\frac{it(\xi-v)^2}{4}}\int e^{\frac{-it^{\frac{1}{2}}u(\xi-v)}{2}}e^{\frac{iu^2}{4}}\chi\left(u\right)\, du.
\end{align*}
Then define $\chi_1(\xi)=e^{i\xi^2}\widehat{[e^{iy^2/2}\chi(y)]}(\xi)$, where by Plancherel's identity we have the property that
\begin{align*}
    \int \chi_1(\xi)\, d\xi\approx \int \chi (y) \, dy=1.
\end{align*}
Then we can write 
\begin{align*}
\hat{\Phi}_v\left(t,\frac{\xi}{2}\right)=t^{\frac{1}{2}}e^{\frac{-it\xi^2}{4}}\chi_1\left(\frac{(\xi-v)t^{\frac{1}{2}}}{2}\right),
\end{align*}
which implies that
\begin{align*}
    \gamma(t,v)=\int e^{it\xi^2}\hat{u}(t,\xi) t^{\frac{1}{2}}\bar{\chi}_1\left(\frac{(2\xi-v)t^{\frac{1}{2}}}{2}\right) \, d\xi.
\end{align*}
By making a change of variables $2\xi\rightarrow \xi,$ we get 

\begin{align}
     \gamma(t,\xi)=&\frac{1}{2}e^{\frac{it\xi^2}{4}}\hat{u}\left(t,\frac{\xi}{2}\right)*_{\xi}t^{\frac{1}{2}}\chi_1\left(\frac{t^{\frac{1}{2}}\xi}{2}\right).
     \label{gammafourierconv}
\end{align}

\subsection{Difference bounds\label{difference}}
Next, we will use these convolutions to prove Lemma \ref{diffLemma}.

\begin{proof}[Proof of Lemma \ref{diffLemma}]

This proof is split up into three parts: The $\gamma$ bounds (equations \eqref{gammabnds} and \eqref{v gamma L^infty}), the spatial difference bounds (\eqref{spatialbnds}, \eqref{spatialL^2}, and \eqref{spatialu_xbnds}), and the Fourier difference bounds (\eqref{fourierbnds1} and \eqref{fourierbnds}). 

\medskip
\textbf{(i) Bounds for $\gamma$:} We can use Young's convolution inequality applied to \eqref{gammaconvolution} to obtain uniform pointwise and $L^2$ bounds:
\begin{align*}
    \|\gamma(t,v)\|_{L^\infty} \leq& t^{\frac{1}{2}}\|w(t,vt)\|_{L^\infty}\|t^{\frac{1}{2}}\chi(vt^{\frac{1}{2}})\|_{L^1}\lesssim t^{\frac{1}{2}}\|u\|_{L^\infty}\\
    \|\gamma(t,v)\|_{L^2_v}\lesssim& t^{\frac{1}{2}}\|w(t,vt)\|_{L^2_v}=\|u\|_{L^2_x}.
\end{align*}
Here we used the fact that the $L^2_x$ and the $L^2_v$
 norms are related as follows:
 \begin{equation*}
     \|f\|_{L^2_x}=t^{\frac{1}{2}}\|f\|_{L^2_v}.
 \end{equation*}
For the third inequality in \eqref{gammabnds}, we use the fact that the derivative of the convolution is
\begin{align*}
    \partial_v t^{-\frac{1}{2}}\gamma=\partial_v w(t,vt)*t^{\frac{1}{2}}\chi(vt^{\frac{1}{2}}).
\end{align*}
Also, by chain rule
\begin{align*}
    \frac{\partial w}{\partial v}= \frac{\partial w}{\partial x} \frac{\partial x}{\partial v}=t\px (e^{-i\phi}u),
\end{align*}
and so we get
\begin{align*}
    t\px (e^{-i\phi}u)=te^{-i\phi}[-i\phi_xu+u_x]=-\frac{i}{2}e^{-i\phi}[xu+2tiu_x]=\frac{-i}{2}e^{-i\phi}Lu.
\end{align*}
Then by Young's inequality, we have
\begin{align}
    \|\partial_v \gamma\|_{L^2_v}\lesssim t^{\frac{1}{2}}\|\partial_v w(t,vt)\|_{L^2_v}=\|Lu\|_{L^2_x}.
    \label{partialgamma}
\end{align}
This finishes the proof of \eqref{gammabnds}. Now we will prove \eqref{v gamma L^infty}. For the first part,
\begin{align*}
    v\gamma =&\int vu e^{-i\phi(t)} \overline{\chi\lp \frac{x-vt}{\sqrt{t}}\rp}\, dx\\
    =&\frac{1}{t}\int (Lu -2ti u_x)e^{-i\phi(t)} \overline{\chi\lp \frac{x-vt}{\sqrt{t}}\rp}\, dx\\
    \|v\gamma\|_{L^\infty}\leq& \frac{1}{t}\|Lu\|_{L^2_x}\left\|\chi\lp \frac{x-vt}{\sqrt{t}}\rp\right\|_{L^2_x}+\|u_x\|_{L^\infty} \int \chi\lp \frac{x-vt}{\sqrt{t}}\rp \, dx
\end{align*}

\begin{align*}
    \left\|\chi\lp \frac{x-vt}{\sqrt{t}}\rp\right\|_{L^2_x}^2 = \int t^{\frac{1}{2}} \left|\chi\lp y\rp\right|^2 \, dy \approx t^{\frac{1}{2}}, \,\,
    \int \chi\lp \frac{x-vt}{\sqrt{t}}\rp \, dx=\int t^{\frac{1}{2}} \chi\lp y\rp \, dy \approx t^{\frac{1}{2}}\\
\end{align*}

So we get 
\begin{align*}
    \|v\gamma\|_{L^\infty}\lesssim t^{-\frac{3}{4}}\|Lu\|_{L^2_x}+ t^{\frac{1}{2}}\|u_x\|_{L^\infty} 
\end{align*}

For the second part, we consider
\begin{align*}
    \frac{d}{dv}|\langle v \rangle^{\frac{k}{2}}\gamma|^2=2\Rea\left(\langle v \rangle^{k}\gamma\gamma_v\right) + k v \, \langle v \rangle^{k-2}|\gamma|^2.
\end{align*}
Then,
\begin{align*}
 \|\langle v \rangle^{\frac{k}{2}}\gamma\|_{L^\infty}^2  \lesssim \int \frac{d}{dv}|\langle v \rangle^{\frac{k}{2}}\gamma|^2\lesssim \|\partial_v \gamma\|_{L^2_v} \|\langle v \rangle^k\gamma\|_{L^2_v} +\|\langle v \rangle^k\gamma\|_{L^2_v}^2 \lesssim \|Lu\|_{L^2_x}\|u\|_{H^k_x} + \|u\|_{H^k_x}^2.
\end{align*}
Then using Young's product inequality we get 
\begin{align*}
 \|\langle v \rangle^{\frac{k}{2}}\gamma\|_{L^\infty}\lesssim&\|Lu\|_{L^2_x}+\|u\|_{H^k_x}
\end{align*}
as needed.

\medskip

\textbf{(ii) Spatial difference bounds:}
Consider
\begin{align*}
     | u(t,vt)-t^{-\frac{1}{2}}e^{i\phi(t,vt)}\gamma(t,v)|=&|w(t,vt)-w(t,vt)*t^{\frac{1}{2}}\chi(yt^{\frac{1}{2}})|\\
    =&\left|\int [w(t,(v-y)t)-w(t,vt)]t^{\frac{1}{2}}\chi(t^{\frac{1}{2}}y)\, dy\right|\\
    \leq& \int \left|w(t,(v-y)t)-w(t,vt)\right|t^{\frac{1}{2}}\left|\chi(t^{\frac{1}{2}}y)\right|\, dy.
\end{align*}
By the Fundamental Theorem of Calculus, 
\begin{align}\label{FTC spatial diff}
    |w(t,vt)-w(t,(v-y)t)|
    =&\int_{v-y}^v\partial_v w(t,tu) \, du\leq |y|^{\frac{1}{2}}\|\partial_v w\|_{L^2}.
\end{align}
So putting it back together, we have
\begin{align*}
     | u(t,vt)-t^{-\frac{1}{2}}e^{i\phi(t,vt)}\gamma(t,v)|\lesssim \|\partial_v w\|_{L^2_v}\int |y|^{\frac{1}{2}}t^{\frac{1}{2}}\left|\chi (t^{\frac{1}{2}}y)\right|\, dy.
\end{align*}
By a change of variables and
equation \eqref{partialgamma}, we have 
\begin{align*}
    | u(t,vt)-t^{-\frac{1}{2}}e^{i\phi(t,vt)}\gamma(t,v)|\lesssim t^{-\frac{3}{4}}\|Lu\|_{L^2_x}.
\end{align*}
This completes the proof of the $L^\infty$ bound \eqref{spatialbnds}. Now for the $L^2$ bound in \eqref{spatialL^2}, we consider the expression 
\begin{align*}
    \int \left|w(t,(v-y)t)-w(t,vt)\right|t^{\frac{1}{2}}\left|\chi(t^{\frac{1}{2}}y)\right|\, dy
\end{align*}
and use the Fundamental Theorem of Calculus to get
\begin{align*}
     \left|w(t,(v-y)t)-w(t,vt)\right|=\left|\int_0^1 y\partial_vw(t,t(v-hy))\, dh\right|,
\end{align*}
so that
 \begin{align*}
   \int \left|w(t,(v-y)t)-w(t,vt)\right|t^{\frac{1}{2}}\left|\chi(t^{\frac{1}{2}}y)\right|\, dy\lesssim& \int \int_0^1 |y|\left|\partial_vw(t,t(v-hy))\right|\, \left|\chi (t^{\frac{1}{2}}y) \right|t^{\frac{1}{2}}\, dh dy.
 \end{align*}
Then, taking the $L^2_v$ norm, we have
\begin{align*}
     \|u(t,vt)-t^{-\frac{1}{2}}e^{i\phi(t,vt)}\gamma(t,v)\|_{L^2_v}\lesssim&\|\partial_vw(t,tv)\|_{L^2_v}\int |y|\left|\chi (t^{\frac{1}{2}}y) \right|t^{\frac{1}{2}}\, dy\\
     \lesssim&t^{-\frac{1}{2}}\|\partial_vw(t,tv)\|_{L^2_v}\lesssim t^{-1}\|Lu\|_{L^2_x}.
\end{align*}
This finishes the proof of the spatial difference bounds \eqref{spatialL^2} for $u$. It remains to prove  the bound in \eqref{spatialu_xbnds}.
For this we consider the following computation 
\begin{align*}
    v\gamma=&\int uve^{-i\phi(t,x)}\widebar{\chi}\left(\frac{x-vt}{\sqrt{t}}\right)\, dx\\
    =&\int \left[\frac{Lu}{t}-2iu_x\right]e^{-i\phi(t,x)}\widebar{\chi}\left(\frac{x-vt}{\sqrt{t}}\right)\, dx\\
   t^{-\frac{1}{2}}v\gamma =&\underbrace{\int \frac{Lu}{t} e^{-i\phi(t,x)}t^{-\frac{1}{2}}\widebar{\chi}\left(\frac{x-vt}{\sqrt{t}}\right)\, dx}_{I}+\int \frac{2}{i}t^{-\frac{1}{2}}u_xe^{-i\phi(t,x)}\widebar{\chi}\left(\frac{x-vt}{\sqrt{t}}\right)\, dx.
\end{align*}
We define $j:=e^{-i\phi}u_x$. Then,
\begin{align*}
    t^{-\frac{1}{2}}v\gamma=\frac{2}{i}j(t,vt)*t^{\frac{1}{2}}\widebar{\chi}(t^{\frac{1}{2}}v)+I,
\end{align*}
where 
\begin{align*}
    |I|\leq \frac{1}{t}\|Lu\|_{L^2}\left\|t^{-\frac{1}{2}}\widebar{\chi}\left(\frac{x-vt}{\sqrt{t}}\right)\right\|_{L^2}\lesssim t^{-\frac{3}{4}}\|Lu\|_{L^2}.
\end{align*}
Then, 
\begin{align*}
   \left|e^{-i\phi}u_x-\frac{i}{2}t^{-\frac{1}{2}}v\gamma\right|\lesssim &|j(t,vt)-j(t,vt)*t^{\frac{1}{2}}\widebar{\chi}(t^{\frac{1}{2}}v)|+t^{-\frac{3}{4}}\|Lu\|_{L^2_x}.
\end{align*}
For the first term, we can argue as in the proof of \eqref{spatialbnds} to obtain
\begin{align*}
    |j(t,vt)-j(t,vt)*t^{\frac{1}{2}}\widebar{\chi}(t^{\frac{1}{2}}v)|\lesssim t^{-\frac{1}{4}}\|\partial_vj\|_{L^2_v}.
\end{align*}
Now we have
\begin{align*}
    \frac{dj}{dv}=t\px(j)=te^{-i\phi}[-i\phi_x u_x+u_{xx}]=-\frac{i}{2}e^{-i\phi}L(u_x),
\end{align*}
and so 
\begin{align*}
    \|\partial_vj\|_{L^2_v}=t^{-\frac{1}{2}}\|L(u_x)\|_{L^2_x},
\end{align*}
completing the proof of \eqref{spatialu_xbnds}.

\medskip

\textbf{(iii) Fourier difference bounds:} First we will prove some initial bounds for 
$\gamma(t,\xi)$ which are similar to \eqref{gammabnds}, using \eqref{gammafourierconv}.  For convenience let $h(t,\xi):=t^{\frac{1}{2}}\chi_1\left(\frac{t^{\frac{1}{2}}}{2}\xi\right)$. 
Using Young's inequality we can get the following bounds:
\begin{align*}
    \|\gamma(t,\xi)\|_{L^\infty}\lesssim& \left\|e^{\frac{it\xi^2}{4}}\hat{u}\right\|_{L^\infty}\left\|h\right\|_{L^1}  \lesssim\|\hat{u}(t,\xi)\|_{L^\infty}\\
    \|\gamma(t,\xi)\|_{L^2_\xi}\lesssim& \left\|e^{\frac{it\xi^2}{4}}\hat{u}\right\|_{L^2_\xi}\left\|h\right\|_{L^1}   \lesssim \|\hat{u}(t,\xi)\|_{L^2_\xi}.
\end{align*}
Also, since
\begin{align*}
    \partial_{\xi}\gamma(t,2\xi)=&\partial_{\xi}\left(e^{it\xi^2}\hat{u}(t,\xi)\right)*h,
    \end{align*}
where
    \begin{align*}
\partial_{\xi}\left(e^{it\xi^2}\hat{u}(t,\xi)\right)
   =&e^{it\xi^2}[2ti\xi \hat{u}(t,\xi)+ \partial_\xi \hat{u}(t,\xi)]
   =-ie^{it\xi^2} \widehat{Lu}(\xi),
\end{align*}
we get
\begin{align*}
    \|\partial_{\xi}\gamma(t,\xi)\|_{L^2_\xi}\approx \|Lu\|_{L^2_x}.
\end{align*}
Now we can start the proof of the difference bounds \eqref{fourierbnds1}, \eqref{fourierbnds}.  Recall that 
\[
\gamma(t,2\xi)=e^{it\xi^2}\hat{u}(t,\xi)*_{\xi}t^{\frac{1}{2}}\chi_1\left(t^{\frac{1}{2}}\xi\right).
\]
First define $g(t,\xi):=e^{it\xi^2}\hat{u}(t,\xi)$, which is the Fourier version of $w$ from before. Then, 
\begin{equation}
    |\hat{u}(t,\xi)-e^{-it\xi^2}\gamma(t,2\xi)|=|g(t,\xi)-\gamma(t,2\xi)|
    =\left|\int t^{\frac{1}{2}}\chi_1(t^{\frac{1}{2}}y) [g(t,\xi)-g(t,\xi-y)]\, dy\right|. \label{fourierdiff}
\end{equation}
Now, using H\"older's inequality, we can write
\begin{align*}
    \left|g(t,\xi)-g(t,\xi-y)\right|=\left|\int_{\xi-y}^\xi \partial_\xi g(t,\xi)\, d\xi\right|
    \lesssim |y|^{\frac{1}{2}}\|Lu\|_{L^2_x}.
\end{align*}
Now,
\begin{align*}
    \|\hat{u}(t,\xi)-e^{-it\xi^2}\gamma(t,2\xi)\|_{L^\infty}\lesssim& \left|\int t^{\frac{1}{2}}\chi_1(yt^{\frac{1}{2}})|y|^{\frac{1}{2}}\|Lu\|_{L^2_x}\right|\\
    \lesssim& t^{-\frac{1}{4}}\|Lu\|_{L^2_x}.
\end{align*}
For the $L^2$ bound, we will use \eqref{fourierdiff} and
\begin{align*}
     |\hat{u}(t,\xi)-e^{it\xi^2}\gamma(t,2\xi)|\leq \left|\int t^{\frac{1}{2}}\chi_1(t^{\frac{1}{2}}y) [g(t,\xi)-g(t,\xi-y)]\, dy\right|,
\end{align*}
where we can write
\begin{align*}
    g(t,\xi)-g(t,\xi-y)=\int_0^1\partial_\xi g(t,\xi-hy)|y|\, dh
\end{align*}
to get
\begin{align*}
    \|g(t,\xi)-\gamma(t,2\xi)\|_{L^2_\xi}\lesssim& \|\partial_\xi g(t,\xi)\|_{L^2_\xi}\int |y|t^{\frac{1}{2}}\chi_1(t^{\frac{1}{2}}y)\, dy\\
    \lesssim& t^{-\frac{1}{2}}\|\partial_\xi g(t,\xi)\|_{L^2_\xi}=t^{-\frac{1}{2}}\|Lu\|_{L^2_x}.
\end{align*}
This completes the proof of \eqref{fourierbnds1} and \eqref{fourierbnds}. 
\end{proof}

\subsection{The asymptotic equation}  \label{ODE}
The aim of this subsection is to prove the result of Lemma~ \ref{ODELemma} which asserts that the asymptotic profile $\gamma$ solves the asymptotic equation \eqref{gamma asym}.

\begin{proof}[Proof of Lemma \ref{ODELemma}]

By a direct computation, we get
\begin{align*}
    i\gamma_t=&\int iu_t\bar{\Phi}_v + iu\bar{\Phi}_{vt}\, dx\\
=&\int -u\overline{\left[i\pt {\Phi}_v+\px^2{\Phi}_v\right]}-i\px(|u|^2u)\bar{\Phi}_v\, dx.
\end{align*}
Now, using \eqref{LinPhi}, we get 
\begin{align*}
i\gamma_t
=&\int \frac{-i}{4t^2}e^{-i\phi}Lu \left[-i(x-vt)\chi(\cdot)+2t^{\frac{1}{2}}\chi'(\cdot)\right]+i(|u|^2u)\px\bar{\Phi}_v\, dx.
\end{align*}
Recall that 
\begin{align*}
    \Phi_v=e^{i\frac{x^2}{4t}}\chi \left(\frac{x-vt}{\sqrt{t}}\right),
\end{align*}
so then
\begin{align*}
    \px{\Phi}_v
    =&t^{-\frac{1}{2}}\Psi_v+\frac{ix}{2t}\Phi_v,
\end{align*}
where $\Psi_v:=e^{i\phi}\chi'(\cdot)$ is very similar to $\Phi_v.$  We can write
\begin{align*}
    i\gamma_t(t)= \, t^{-1}\frac{v}{2}(|\gamma|^2\gamma)-R(t,v)
\end{align*}
 where
\begin{align*}
    R(t,v)=:&\int \frac{i}{4t^2}e^{-i\phi}Lu \left[-i(x-vt)\chi(\cdot)+2t^{\frac{1}{2}}\chi'(\cdot)\right]\, dx-\int i(|u|^2u)\px\bar{\Phi}_v\, dx +t^{-1}\frac{v}{2}(|\gamma|^2\gamma).
\end{align*}
We can split the reminder term further  and estimate each piece
\begin{align*}
    R(t,v)=R_1+R_2+R_3+R_4
\end{align*}
with
\begin{align*}
    R_1=&\int \frac{i}{4t^2}e^{-i\phi}Lu \left[-i(x-vt)\chi(\cdot)+2t^{\frac{1}{2}}\chi'(\cdot)\right]\, dx\\
    R_2=&-\int i(|u|^2u)t^{-\frac{1}{2}}\bar{\Psi}_v\\
    R_3=&-\int \frac{v}{2}u\bar{\Phi}_v[|u|^2-|u(t,vt)|^2]\, dx\\
    R_4=&\frac{v}{2}\, \gamma [t^{-1}|\gamma|^2-|u(t,vt)|^2].
\end{align*}
We will first prove \eqref{Rbound0}. $R_1$ can be bounded as follows.
We can write it as a convolution: 
\begin{align*}
    R_1=\frac{1}{2t}\left[\left(it^{\frac{1}{2}}\chi'(t^{\frac{1}{2}}v)+tv\chi(t^{\frac{1}{2}}v)\right)*_v\partial_v w(t,vt)\right]
\end{align*}
Then we can get the bound using H\"older's inequality:
\begin{align}
    |R_1|\lesssim t^{-\frac{3}{4}}\|\partial_v w(t,vt)\|_{L^2_v}=t^{-\frac{5}{4}}\|Lu\|_{L^2_x}.
    \label{R1-est}
\end{align}
$R_2$ can be bounded as follows:
\begin{align*}
    |R_2|\lesssim \|u\|_{L^\infty}^3 t^{-\frac{1}{2}}\int |\bar{\Psi}_v|\, dx \lesssim \|u\|_{L^\infty}^3.
\end{align*}
Now we estimate $R_3$. 
We have
\begin{align*}
    |R_3|\lesssim&t^{-1}\|u\|_{L^\infty}\int xu\bar{\Phi}_v[u(t,x)-u(t,vt)]\, dx=t^{-1}\|u\|_{L^\infty}\int \lp Lu-2tiu_x\rp\bar{\Phi}_v[u(t,x)-u(t,vt)]\, dx.
\end{align*}
The first term can be bounded by 
\begin{align*}
    t^{-1}\|u\|_{L^\infty}^2 \int Lu \bar{\Phi}_v\leq  t^{-\frac{3}{4}}\|u\|_{L^\infty}^2 \|Lu\|_{L^2_x}.
\end{align*}
For the second term, by doing a change of variables $x=t(z+v)$, then proceeding as for the spatial difference bounds in \eqref{FTC spatial diff}, we get this term is bounded by
\begin{align*}
    \|u\|_{L^\infty} \|u_x\|_{L^\infty}t^{-\frac{1}{4}}\|Lu\|_{L^2_x}.
\end{align*}
For $R_4$, we can write 
\begin{align*}
     |t^{-1}|\gamma|^2-|u(t,vt)|^2|=&\left||u(t,vt)|^2-|t^{-\frac{1}{2}}\gamma||w(t,vt)*t^{\frac{1}{2}}\chi(vt^{\frac{1}{2}})|\right|\\
     \lesssim& \|u\|_{L^\infty }\left||u(t,vt)|-|w(t,vt)*t^{\frac{1}{2}}\chi(vt^{\frac{1}{2}})|\right|.
\end{align*}
Combining this with the difference bound \eqref{spatialbnds} and \eqref{v gamma L^infty}, we have 
\begin{align*}
    |R_4|\lesssim \|v\gamma\|_{L^\infty}\|u\|_{L^\infty}t^{-\frac{3}{4}}\|Lu\|_{L^2_x}\lesssim t^{-\frac{3}{4}}\|u\|_{L^\infty}\|Lu\|_{L^2_x}\lp t^{-\frac{3}{4}}\|Lu\|_{L^2_x}+ t^{\frac{1}{2}}\|u_x\|_{L^\infty}\rp.
\end{align*}
Overall this gives us the bound
\begin{align*}
    \|R\|_{L^\infty}\lesssim & \ t^{-\frac{5}{4}}\|Lu\|_{L^2_x}+\|u\|_{L^\infty}^3 +t^{-\frac{3}{4}}\|u\|_{L^\infty}^2 \|Lu\|_{L^2_x}+  \|u\|_{L^\infty} \|u_x\|_{L^\infty}t^{-\frac{1}{4}}\|Lu\|_{L^2_x}\\
    &+t^{-\frac{3}{4}}\|u\|_{L^\infty}\|Lu\|_{L^2_x}\lp t^{-\frac{3}{4}}\|Lu\|_{L^2_x}+ t^{\frac{1}{2}}\|u_x\|_{L^\infty}\rp.
\end{align*}
Next we will prove \eqref{Rbound1}. We will consider two cases, when $|v|<1$ and when $|v|\geq 1$.  First, for the case when $|v|<1,$ we already have a bound, using \eqref{R1-est}:
\begin{align*}
    |vR_1|\leq |R_1|\lesssim  t^{-\frac{5}{4}}\|Lu\|_{L^2_x}.
\end{align*}
For the case $|v|\geq 1,$ we  prove a bound with a $\langle v\rangle^{k} $ weight  and interpolate between it and \eqref{R1-est}.
We have
\begin{align}
    R_1=&\int u\frac{e^{-i\phi}}{2t}\px\left[i(x-vt)\chi(\cdot)+2t^{\frac{1}{2}}\chi’(\cdot)\right]\, dx .
    \label{R1bound}
\end{align}
Now define $I:=[-\sqrt{t}+vt, \sqrt{t}+vt]$. Notice that for $\chi$ supported on $[-1,1]$, $\chi\left(\frac{x-vt}{\sqrt{t}}\right)$ is supported on $I$.
Then assuming that $|v| \geq 1$ we will have $|x| \approx |v||t|$ in $I$, and so we can write

\begin{align}
    vR_1=&\int \lp \frac{x}{t}\rp u\frac{e^{-i\phi}}{2t}\px\left[i(x-vt)\chi(\cdot)+2t^{\frac{1}{2}}\chi’(\cdot)\right]\, dx \\
    =&\frac{1}{t}\int 
    \lp Lu +2ti u_x\rp\frac{e^{-i\phi}}{2t}\px\left[i(x-vt)\chi(\cdot)+2t^{\frac{1}{2}}\chi’(\cdot)\right]\, dx .
\end{align}
Since 
\begin{align*}
\px\left[i(x-vt)\chi(\cdot)+2t^{\frac{1}{2}}\chi’(\cdot)\right]=O(1)\,\, \text{ in } t,
\end{align*}
the first term can be bounded by

\begin{align*}
    t^{-2}\|Lu\|_{L^2(I)}\|e^{-i\phi}\|_{L^2(I)}\leq t^{-\frac{7}{4}}\|Lu\|_{L^2(I)}.
\end{align*}
For the second term, we subtract
\begin{align*}
    \frac{1}{t}u_x(t,vt)e^{-i\phi(t,vt)}\int \px\left[i(x-vt)\chi(\cdot)+2t^{\frac{1}{2}}\chi’(\cdot)\right]\, dx =0
\end{align*}
to get that the second term is equal to 
\begin{align*}
    \frac{i}{t}\int 
     \lp u_x  e^{-i\phi}-u_x(t,vt)e^{-i\phi(t,vt)} \rp \px\left[i(x-vt)\chi(\cdot)+2t^{\frac{1}{2}}\chi’(\cdot)\right]\, dx .
\end{align*}
Then we can proceed similarly to as in the proof of the spatial difference bounds. Define $w_2:= u_x  e^{-i\phi}$, and we get $\p_v w_2=-\frac{i}{2}e^{-i\phi} Lu_x$. Then, using the change of variables $x=t(z+v)$ and the Fundamental Theorem of Calculus, we get
\begin{align*}
    &\frac{i}{t}\int 
     \lp w_2(t, (v+z)t)-w_2(t,vt) \rp \p_z\left[izt\chi(zt^{\frac{1}{2}})+2t^{\frac{1}{2}}\chi’(zt^{\frac{1}{2}})\right]\, dz \\
     \lesssim&\frac{1}{t}\|\p_v w_2\|_{L^2_v} \int z^{\frac{1}{2}} \p_z\left[izt\chi(zt^{\frac{1}{2}})+2t^{\frac{1}{2}}\chi’(zt^{\frac{1}{2}})\right]\, dz \\
      \lesssim&t^{-\frac{3}{2}}\|Lu_x\|_{L^2_x} \int \lp \frac{x-vt}{t} \rp^{\frac{1}{2}} \p_x\left[i(x-vt)\chi(\cdot)+2t^{\frac{1}{2}}\chi’(\cdot)\right]\, dx\\
         \lesssim&t^{-\frac{3}{2}}\|Lu_x\|_{L^2_x} \int_I \lp \frac{x-vt}{t} \rp^{\frac{1}{2}} \, dx=t^{-\frac{5}{4}}\|Lu_x\|_{L^2_x}.
\end{align*}
Overall, this gives
\begin{align*}
    \|vR_1\|_{L^\infty}\lesssim t^{-\frac{5}{4}}\lp\|Lu\|_{L^2_x} +\|Lu_x\|_{L^2_x}  \rp.
\end{align*}
Next, we have
\begin{align*}
    |vR_2|=\left|\int ivu|u|^2t^{-\frac{1}{2}}\bar{\Psi}_v\right|=\left|\int i\frac{x}{t} u(|u|^2)t^{-\frac{1}{2}}\bar{\Psi}_v\right|=\left|\int \left(\frac{iLu}{t}+2u_x\right) |u|^2t^{-\frac{1}{2}}\bar{\Psi}_v\right|.
\end{align*}
Then the first term is
\begin{align*}
    \left|\int t^{-\frac{3}{2}}Lu|u|^2\widebar{\Psi}_v\, dx\right|\lesssim t^{-\frac{3}{2}}\|u\|_{L^\infty}^2\|Lu\|_{L^2}\|\Psi_v\|_{L^2},
\end{align*}
where $\|\Psi_v\|_{L^2}\approx t^\frac{1}{4}$. The second term is 
\begin{align*}
    \left|\int u_x|u|^2t^{-\frac{1}{2}}\widebar{\Psi}_v\, dx\right|\lesssim \|u_x\|_{L^\infty}\|u\|_{L^\infty}^2\int t^{-\frac{1}{2}}\widebar{\Psi}_v\, dx\approx \|u_x\|_{L^\infty}\|u\|_{L^\infty}^2.
\end{align*}
For $vR_3,$ by multiplying the previous $R_3$ bound by $v$, we get 
\begin{align*}
    |vR_3|\lesssim&t^{-2}\|u\|_{L^\infty}\int xu\bar{\Phi}_v[xu(t,x)-vtu(t,vt)]\, dx\\
    =&t^{-2}\|u\|_{L^\infty}\int \lp Lu-2tiu_x\rp e^{-i\phi} \chi\lp\frac{x-vt}{\sqrt{t}} \rp[Lu(t,x)-Lu(t,vt)-2tiu_x(t,x)+2tiu_x(t,vt)]\, dx\\
    \lesssim& \|u\|_{L^\infty}\left[ t^{-2}\|Lu\|_{L^2}^2\int \chi(\cdot) +t^{-1}\|u_x\|_{L^\infty}\|Lu\|_{L^2_x}\|\chi(\cdot)\|_{L^2_x}+ \|u_x\|_{L^\infty}\|\p_v w_2\|_{L^2_v}\int tz^{\frac{1}{2}}\chi(t^{\frac{1}{2}}z)\, dz \right]\\
     \lesssim& \|u\|_{L^\infty}\left[ t^{-\frac{3}{2}}\|Lu\|_{L^2}^2+t^{-\frac{3}{4}}\|u_x\|_{L^\infty}\|Lu\|_{L^2_x}+ \|u_x\|_{L^\infty}t^{-\frac{1}{2}}\|Lu_x\|_{L^2_x}t^{\frac{1}{4}}\right],
\end{align*}
where we used that $\|\chi(\cdot)\|_{L^1_x}\sim t^{\frac{1}{2}}, \|\chi(\cdot)\|_{L^2_x}\sim t^{\frac{1}{4}}$, and for the last term, we used the same method as in the bound for $vR_1.$

Similarly for $R_4$ we can get
\begin{align*}
|vR_4|\lesssim&\|v\gamma\|_{L^\infty}\|u\|_{L^\infty}(t^{-\frac{1}{2}}v \gamma -v u(t,vt)).
\end{align*}
Notice that $vw(t,vt)=\frac{1}{t}e^{-i\phi}Lu(t,vt)-2i w_2(t,vt)$, so then

\begin{align*}
    t^{-\frac{1}{2}}v \gamma -v u(t,vt)=&v w(t,vt)* t^{\frac{1}{2}}\chi(v t^{\frac{1}{2}})-vu(t,vt)\\
    =&\int \frac{1}{t}\lp e^{-i\phi((v-y)t}Lu(t,(v-y)t)-e^{-i\phi(vt)}Lu(t,vt) \rp t^{\frac{1}{2}} \chi(t^{\frac{1}{2}} y) \, dy\\
    -&2i \int \lp w_2(t,(v-y)t)- w_2(t,vt)\rp t^{\frac{1}{2}} \chi(t^{\frac{1}{2}} y) \, dy\\
    \lesssim& \lp t^{-1}\|\p_v Lu\|_{L^2_v}+ \|\p_v w_2\|_{L^2_v} \rp \int |y|^{\frac{1}{2}}t^{\frac{1}{2}}\chi(t^{\frac{1}{2}} y) \, dy \\
    \lesssim& t^{-\frac{5}{4}}\|(Lu)_x\|_{L^2_x}+t^{-\frac{3}{4}}\|Lu_x\|_{L^2_x}.
\end{align*}

This completes the proof of \eqref{Rbound1}.
\end{proof}

\section{The bootstrap argument and the conclusion of the proof \label{bootstrapsection}}
In this section  we use a bootstrap argument in order to complete the proof of Theorem~\ref{t:main}. Recall that we make two bootstrap assumptions, on $u$ and its derivative as follows, see \eqref{bootstrap}:
\begin{align*}
    \|u\|_{L^\infty}\leq D\epsilon \langle t\rangle^{-\frac{1}{2}},\ \qquad
      \|u_x\|_{L^\infty}\leq D\epsilon \langle t\rangle^{-\frac{1}{2}},
\end{align*}
where $D$ is a sufficiently large universal constant to be chosen later. These bootstrap assumptions are assumed to hold for a solution $u$ to the DNLS equation in a time interval $[0,T]$ with $T$ arbitrary large. Then our objective will be to improve the constant $D$ in these bounds under the assumption that $\epsilon $ is sufficiently small. Once this is achieved, a standard continuity argument shows that the solution $u$ is global in time and satisfies these bounds, thereby concluding the proof of the theorem.

\subsection{Energy bounds for $u$ \label{boundsU}}  The energy bounds for $u$ are given in Lemma~\ref{l:energy est u} These are relatively standard and follow from the well-posedness, since the data-to-solution map is continuous.  
\begin{proof}[Proof of Lemma~\ref{l:energy est u}]
 For $k\geq \frac{1}{2}$, we have Lipschitz continuity of the data-to-solution map \cite{cite:Taks12}, so in particular, we have
    \begin{align*}
       \|u\|_{H^k}\leq C\|u_0\|_{H^k} \lesssim \epsilon.
    \end{align*}
The case when $0\leq k\le \frac{1}{2}$ also follows from the global well-posedness of \eqref{dnls} as seen in \cite{cite:1} and the references within.

\end{proof}

\begin{rmk}
The conservation laws for \eqref{dnls} are a consequence of the complete integrability. However having these conservation laws, strictly speaking, are not necessary for our arguments. A simple, robust, and more direct energy estimate using Gr\"onwall's inequality and the above bootstrap bounds would give an estimate of the following form, which will suffice for the arguments below
\[
\Vert u(t)\Vert_{H^k} \lesssim \epsilon \langle t\rangle^{\frac{D^2\epsilon^2}{2}}.
\]
\end{rmk}

\subsection{Energy bounds for $Lu$} \label{boundsLU}
In this subsection, we prove energy bounds on $Lu$ and $L(u_x)$ given in Lemma~\ref{LuBound}.
We can verify the following identities. 
\begin{align*}
    [L,\px]=1,\,[i\pt+\px^2,L]=0, \,L(|u|^2u)=2|u|^2 Lu-u^2\widebar{Lu}.
    \end{align*}
Therefore, by applying $L$ to \eqref{dnls} and using the identities, we get
\begin{align*}
   L( iu_t+u_{xx})=&(i\pt+\px^2)Lu=-iL(\partial_x(|u|^2u))
   =-i[\px(2|u|^2Lu-u^2\widebar{Lu})+|u|^2u],
\end{align*}
which gives us an equation for $Lu$.
This is similar to the linearization of \eqref{dnls} so we expect it to have well-posedness and good bounds on the solution. For convenience, we will rewrite this in terms of the variable $z:=Lu$ as
\begin{align*}
    (i\pt+\px^2)z=-i[\px(2|u|^2z-u^2\widebar{z})+|u|^2u].
\end{align*}
By doing energy estimates in the usual way, we get
\begin{align*}
    \frac{d}{dt}\|z\|_{L^2}^2
    =&2\int |u|^2 \px (|z|^2) \, dx +\Rea \int u^2\px (\bar{z}^2)\, dx-2\Rea \int |u|^2 u\bar{z}\\
    =&-2\int \px (|u|^2)  (|z|^2) \, dx -\Rea \int \px( u^2) \bar{z}^2\, dx-2\Rea \int |u|^2 u\bar{z}\, dx.
\end{align*}
Then,
\begin{align*}
    \left|\frac{d}{dt}\|z\|_{L^2}^2\right|\lesssim& \|u\|_{L^\infty}\|u_x\|_{L^\infty}\|z\|_{L^2}^2+\|u\|_{L^\infty}^2\|u\|_{L^2}\|z\|_{L^2}\\
\leq&(D\epsilon  \langle t\rangle^{-\frac{1}{2}})^2 \|z\|_{L^2}^2+(D\epsilon  \langle t\rangle^{-\frac{1}{2}})^2\epsilon \|z\|_{L^2}.
\end{align*}
Then by letting $y:=\|z\|_{L^2}^2$, we can solve the following Bernoulli equation:
\begin{align*}
    y'+p(t)y\leq B(t)y^{\frac{1}{2}},
\end{align*}
where $p(t)=-D^2\epsilon^2\langle t\rangle^{-1}, B(t)=D^2\epsilon^3 \langle t\rangle^{-1}$.
By making the substitution $w=y^{\frac{1}{2}}$, we get
\begin{align*}
    w'+\frac{p(t)}{2}w\leq \frac{B(t)}{2},
\end{align*}
 which is a linear equation, so solving this and using the initial condition, i.e. the localization $\|Lu(0)\|_{L^2}=\|xu_0\|_{L^2}\leq \epsilon$, we get
\begin{align*}
     \|z\|_{L^2}\lesssim \epsilon \,\langle t\rangle^{\frac{D^2\epsilon^2}{2}}.
\end{align*}
In terms of $Lu$, this is
\begin{align*}
    \|Lu\|_{L^2_x}\lesssim \,\epsilon\,\langle t\rangle ^{\frac{D^2\epsilon^2}{2}}.
\end{align*}
We can prove the bound for $\|L(u_x)\|_{L^2_x}$ similarly since $z_2:=L(u_x)$
satisfies the following equation:
\begin{align*}
    (i\pt+\px^2)z_2
    =&-i\px\left[2|u|^2z_2+2\px(|u|^2)Lu-2uu_x\widebar{Lu}-u^2z_2\right]-i(2u_x|u|^2+u^2\bu_x).
\end{align*}
Then, by similar computations of the energy estimate, we end up with the same bound:
\begin{align*}
        \|L(u_x)\|_{L^2_x}\lesssim \epsilon \, \langle t\rangle^{\frac{D^2\epsilon^2}{2}}.
\end{align*}
Note that since $Lu_x=\px(Lu)+u, $ we get the same bound for $\px(Lu)$  in $L^2_x.$

\subsection{Completing the proof} \label{completeproof}
Now we will prove Lemma~\ref{finishingprooflemma}, i.e. the bounds with better constants, in order to close the bootstrap argument and thus complete the proof of the main theorem, Theorem~\ref{t:main}. First, we will close the bootstrap on $u$.
Using Lemma \ref{diffLemma} and Lemma \ref{LuBound}, we get
\begin{align}
\label{dif}
    \|e^{-i\phi}u-t^{-\frac{1}{2}}\gamma\|_{L^\infty}\lesssim t^{-\frac{3}{4}}\|Lu\|_{L^2_x}\lesssim \epsilon\, \langle t\rangle ^{\frac{D^2\epsilon^2}{2}-\frac{3}{4}}, \qquad t\gtrsim 1.
\end{align}
Using the second part of \eqref{v gamma L^infty} with $k=0$ and the energy estimates for $u$ and $Lu$, we have
\begin{align*}
    \|\gamma(1,v)\|_{L^\infty}\lesssim \epsilon.
\end{align*}
By applying the bootstrap assumptions \eqref{bootstrap} to \eqref{Rbound0}, we get
\begin{align*}
    \|R\|_{L^\infty}\lesssim t^{-\frac{5}{4}}\epsilon\, \langle t\rangle^{\frac{D^2\epsilon^2}{2}}+D^3\epsilon^3\,\langle t\rangle^{-\frac{3}{2}}+2D^2\epsilon^2\,\langle t\rangle^{-\frac{5}{4}}\epsilon\, \langle t\rangle^{\frac{D^2\epsilon^2}{2}}+2D\epsilon\,  \langle t\rangle^{-2}\epsilon^2\, \langle t\rangle^{D^2\epsilon^2}.
\end{align*}
By doing an energy estimate on our asymptotic equation
\begin{align*}
    i\gamma_t(t,v)= \, t^{-1}\frac{v}{2}(|\gamma (t,v)|^2\gamma(t,v))-R(t,v),
\end{align*} 
we get
\begin{align*}
   \int_1^t \frac{d}{ds}|\gamma(s,v)|^2\, ds=& 2\Rea\int_1^t iR(s,v)\bar{\gamma}(s,v)\, ds,
\end{align*}
which leads to
\begin{equation}
\label{gammaeqn}
    \left|   \gamma(t,v)\right|\lesssim |\gamma(1,v)|+ \int_1^t|R(s,v)|\, ds. 
\end{equation}
The integral can be bounded by
\begin{align*}
    \int_1^t| R(s)|\, ds\lesssim \epsilon+ D^3\epsilon^3+2D^2\epsilon^3+2D\epsilon^3,
\end{align*}
where we need $D^2\epsilon^2< \frac{1}{4}$ so that the terms are all integrable, i.e., $D\epsilon<\frac{1}{2}$.
Then, we have
\begin{align*}
   | \gamma(t,v)|\lesssim \, & \epsilon +\epsilon+ D^3\epsilon^3+2D^2\epsilon^3+2D\epsilon^3\lesssim \epsilon,
\end{align*}
for small enough $\epsilon$. 

Combining this bound with \eqref{dif},  by the triangle inequality we obtain
\begin{align*}
    |u|\leq& \, t^{-\frac{1}{2}}|\gamma|+|e^{-i\phi}u-t^{-\frac{1}{2}}\gamma|\\
\lesssim&\,  \epsilon t^{-\frac{1}{2}}[1+D+\epsilon].
\end{align*}
Then we need 
\begin{align*}
    1+\epsilon \ll D \ll \frac{1}{\epsilon},
\end{align*}
which is possible if we choose $\epsilon$ small enough. This closes the first part of the bootstrap argument. 
To close the $u_x$ bootstrap bound, we will do the same argument as for $u.$
By triangle inequality, we can write
\begin{align}
    |u_x|\lesssim t^{-\frac{1}{2}}|v\gamma|+|u_x(t,vt)-t^{-\frac{1}{2}}e^{i\phi(t,vt)}v\gamma(t,v)|.
    \label{u_xtriangle}
\end{align}
For the first term, we have from \eqref{gammaeqn} that 
\begin{equation*}
    |v\gamma(t,v)|\lesssim |v\gamma(1,v)|+ \int_1^t|vR(s,v)|\, ds. 
\end{equation*}
From  \eqref{v gamma L^infty}, we have 
\begin{align*}
\|v\gamma(0,v)\|_{L^\infty}\lesssim& \left(\|Lu(1)\|_{L^2}+D\epsilon\right)\lesssim (1+D)\epsilon.
\end{align*}

From \eqref{Rbound1} and using the bootstrap assumption \eqref{bootstrap}, we have
\begin{align*}
  |vR|
    \lesssim& \epsilon \langle t\rangle^{\frac{D^2\epsilon^2}{2}-\frac{5}{4}}+D^2\epsilon^3\langle t\rangle^{\frac{D^2\epsilon^2}{2}-\frac{5}{4}}+\epsilon^3D^3\langle t\rangle^{-\frac{3}{2}}+D\epsilon^3\langle t\rangle^{-\frac{5}{4}}
\end{align*}
Note that in order to integrate this in time we need to assume $\frac{D^2\epsilon^2}{2}<\frac{1}{4}$. Then,
\begin{align*}
    \int_1^t |vR(s,v)|\, ds\lesssim&\,  \epsilon+D^3\epsilon^3+D^2\epsilon^3+D\epsilon^3,
\end{align*}
and so
\begin{align*}
    |v\gamma(t,v)|\lesssim 2\epsilon+ \epsilon+D^3\epsilon^3+D^2\epsilon^3+D\epsilon^3 \lesssim \epsilon
\end{align*}
for $\epsilon$ small enough.
Now, putting this all back into \eqref{u_xtriangle}, we have
\begin{align*}
    |u_x|
    \lesssim&\,  \epsilon \langle t \rangle ^{-\frac{1}{2}}\left[ 5+D^3\epsilon^2+D^2\epsilon^2+D\epsilon^2\right].
\end{align*}
Just as before, if $\epsilon$  is small enough then we have $ 1+\epsilon\ll D\ll \frac{1}{\epsilon}$, which suffices.
Now that we have closed both bootstrap assumptions, we have completed the proof of Theorem \ref{t:main}.

\bibliography{refs}{}
\bibliographystyle{plain}

\end{document}